\documentclass[10pt]{article}
\usepackage{latexsym, amssymb, bbm}
\usepackage{amsthm}
\usepackage{xcolor}
\usepackage{graphicx}
\usepackage{amsmath}
\usepackage{enumerate}

\newtheorem{lem}{Lemma}[section]
\newtheorem{thm}[lem]{Theorem}
\newtheorem{prop}[lem]{Proposition}
\newtheorem{cor}[lem]{Corollary}

\newtheorem{remark}[lem]{Remark}

\theoremstyle{definition}
\newtheorem{defn}[lem]{Definition}

\DeclareMathOperator{\Diff}{Diff}
\DeclareMathOperator{\id}{id}
\DeclareMathOperator{\Int}{Int}

\newcommand{\abs}[1]{\left| #1 \right|}

\newcommand{\ide}{\boldsymbol{\mathbbm{1}}}

\newcommand{\bthm}{\begin{thm}}
\newcommand{\ethm}{\end{thm}}

\newcommand{\blem}{\begin{lem}}
\newcommand{\elem}{\end{lem}}

\newcommand{\bcor}{\begin{cor}}
\newcommand{\ecor}{\end{cor}}

\newcommand{\bprop}{\begin{prop}}
\newcommand{\eprop}{\end{prop}}

\newcommand{\brmk}{\begin{remark}}
\newcommand{\ermk}{\end{remark}}

\newcommand{\bpf}{\begin{proof}}
\newcommand{\epf}{\end{proof}}

\newcommand{\beq}{\begin{equation}}
\newcommand{\eeq}{\end{equation}}

\numberwithin{equation}{section}

\def\C{\mathbb{C}}

\def\Q{\mathbb{Q}}
\def\R{\mathbb{R}}
\def\Z{\mathbb{Z}}

\def\CP{\mathbb{CP}}
\def\RP{\mathbb{RP}}
\def\bP{\mathbb{P}}

\def\cL{\mathcal{L}}

\def\cN{\mathcal{N}}

\def\d{\partial}

\def\eps{\epsilon}
\def\i{\iota}

\def\l{\lambda}

\def\w{\omega}

\def\Symp{\mbox{Symp}}

\def\xkm2{\overline{X}_{k-2}}

\begin{document}

\title{Spherical Lagrangians via ball packings and symplectic cutting}

\author{\textsc Matthew Strom Borman\thanks{Partially supported by NSF-grant DMS 1006610.},
Tian-Jun Li\thanks{Supported by NSF-grant DMS 0244663.},
and Weiwei Wu\footnotemark[2]}

\maketitle

\begin{abstract}
In this paper we prove the connectedness of symplectic ball packings
in the complement of a spherical Lagrangian, $S^{2}$ or $\RP^{2}$,
in symplectic manifolds that are rational or ruled.
Via a symplectic cutting construction this is a natural extension of
McDuff's connectedness of ball packings in other settings
and this result has applications to several different
questions: smooth knotting and unknottedness results for
spherical Lagrangians, the transitivity of the action of the
symplectic Torelli group, classifying Lagrangian isotopy classes in the presence of knotting,
and detecting Floer-theoretically
essential Lagrangian tori in the del Pezzo surfaces.
\end{abstract}

%


\section{Introduction}

In \cite{LW12} the second and third named authors investigated the existence and uniqueness
 (unknottedness) problems of
Lagrangian $S^2$ in rational manifolds via tools from symplectic field theory and
the study of symplectic ball-packings \cite{MP94,Mc98}.
In this paper we continue to explore the connections between symplectic ball
packing, symplectic cutting, and Lagrangian unknottedness, while
answering several questions from \cite{LW12} and extending the results to Lagrangian $\RP^2$s.
See \cite{EP97} for an early survey of the problem of Lagrangian knots and more recent results in
\cite{DRE12, FS04, Ev10, Ev11, Hi04, Se08, Vi06}.

Our first result,
conjectured in \cite[Remark 5.2]{LW12},
is on the connectedness of symplectic ball packings in the complement of a
Lagrangian $S^{2}$ or $\RP^{2}$:

\begin{thm}\label{t:BPRL-sphere}
    Let $(M^{4}, \w)$ be a closed $4$-dimensional symplectic manifold
    that is rational or ruled
    and let
    $L \subset M$ be a Lagrangian $S^{2}$
    or $\RP^{2}$,
    then the space of symplectic ball packings in $M \backslash L$ is connected.
\end{thm}
\noindent
Via symplectic
cutting, Theorem~\ref{t:BPRL-sphere}
follows from the connectedness of
the space of ball packings in the complement of a symplectic sphere.
This is established for certain symplectic surfaces in
Proposition \ref{p:BPRS} by work of McDuff on relative inflation \cite{Mc98, Mc12}.

Recall that a symplectic rational manifold $(M^{4}, \w)$ is where $M$ is either $\CP^{2}$, a symplectic blow-up of $\CP^2$, or $S^2\times S^2$.  In \cite{LW12}, building on work of Evans \cite{Ev10}, Hamiltonian unknottedness for Lagrangian $S^2$'s in symplectic rational manifolds was established
when the Euler characteristic $\chi(M) \leq 7$,
except for the case of a characteristic homology class.
We complete the picture here in Theorem~\ref{t:char} in the appendix.
This result is sharp due to Seidel's \cite{Se08} construction of Hamiltonian knotted Lagrangian spheres in 
$\CP^{2}\#5\overline{\CP}^{2}$.

As noted in \cite[Section 6.4.2]{LW12},
Theorems~\ref{t:BPRL-sphere} and \ref{t:char} have the following consequences,
which we also prove in the appendix.
Recall that the symplectic Torelli group
$\Symp_{h}(M, \w)$ is the subgroup of $\Symp(M, \w)$ that acts trivially on homology $H_{*}(M;\Z)$.

\begin{cor}\label{c:TTS2} Suppose  $(M^{4}, \w)$ is a symplectic rational manifold.

 \begin{enumerate}[(1)]
  \item (Symplectic unknottedness of Lagrangian $S^2$)
 The symplectic Torelli group
     $\Symp_{h}(M, \w)$ acts transitively on homologous Lagrangian spheres.

  \item (Smooth unknottedness of Lagrangian $S^2$)
Homologous
    Lagrangian spheres are smoothly isotopic to each other.
  \end{enumerate}
\end{cor}
\noindent
The smooth unknottedness for Lagrangian spheres was first noticed by Evans \cite{Ev10} in the case
of del Pezzo surfaces, i.e.\ monotone rational symplectic manifolds.

Let $(X_{5}, \w_{0})$ be a monotone $\CP^{2}\#5\overline{\CP}^{2}$.   
By \cite[Theorem 1.4]{LW12} one can explicitly classify the homology classes $\xi \in H_{2}(X_{5}; \Z)$
that can be represented by a Lagrangian sphere in $(X_{5}, \w_{0})$.  Furthermore
Evans in \cite[Theorem 1.3]{Ev11} computes the weak homotopy type of
$\Symp_{h}(X_{5}, \w_{0})$ and in \cite[Section 6.1]{Ev11} shows $\pi_{0}(\Symp_{h}(X_{5}, \w_{0}))$
is a $\Z_{2}$-quotient of the pure braid group $PBr(S^{2}, 5)$ on $S^{2}$ with $5$ strands.
These two results above together with Corollary \ref{c:TTS2} show that Lagrangian spheres in 
$(X_{5}, \w_{0})$ are unique up to Hamiltonian isotopy and a certain (explicit) braid group action.  

Corollary \ref{c:TTS2} therefore allows for the first explicit description of Hamiltonian
isotopy classes of Lagrangian spheres in a closed symplectic manifold where there is
Lagrangian knotting.  Hind \cite{Hi12} has done this in the non-compact setting with
the plumbing of two $T^{*}S^{2}$.  In forthcoming work of the third author \cite{Wu13} this braid group
action will be connected with
Lagrangian Dehn twists along a finite set of Lagrangian spheres and the
Hamiltonian isotopy classes of Lagrangian spheres in an $A_{n}$-singularity will be studied.

Such an explicit description brings up a more intriguing question.
For a symplectic rational manifold Corollary \ref{c:TTS2}(1) tells us we have a transitive
group action of $\pi_{0}(\Symp_{h}(M)) = \Symp_{h}(M)/\Symp_{0}(M)$
on $\mathcal{L}ag_{\xi}(M, S^{2})$, the Hamiltonian isotopy classes of Lagrangian spheres in the class
$\xi \in H_{2}(M; \Z)$, and the stabilizer of this action seems hard to understand.
It may be possible to understand the stabilizer of the action of 
$\pi_{0}(\Symp_{h}(M))$ on the Fukaya category in terms of braid group elements \cite{KS02}.

The unknotting picture for Lagrangian $\RP^2$ is more intriguing.
We first have the following parallel results of Corollary \ref{c:TTS2} for small Betti numbers:

\bthm\label{t:TTRP2} Let $(M^4,\w)$ be a symplectic rational manifold and $b^{-}_{2}(M) \leq 8$.
   \begin{enumerate}[(1)]
   \item  $\Symp_h(M, \w)$ acts transitively on $\Z_2$-homologous Lagrangian $\RP^2$'s.
   \item  $\Z_2$-homologous Lagrangian $\RP^2$'s are smoothly isotopic.
   \end{enumerate}
\ethm
\noindent
Note in Lemma~\ref{l:RP2 classes} we prove there is an unique
$\Z_{2}$-homology class containing a Lagrangian $\RP^{2}$ when $b^{-}_{2}(M) \leq 2$.

While it is still possible that
the uniqueness of $\Z_2$-homologous $\RP^2$'s up to smooth isotopy
 is valid for an arbitrary symplectic rational manifold,
 the following
 result  gives some hints about the complication of
the problem:

\begin{prop} \label{p:KRP2}
 Let  $M=\CP^2\#k \overline{\CP}^{2}$.

 \begin{enumerate}[(1)]

 \item  For $k \geq 9$ there exists a symplectic form $\w$ on $M$
 with $L$ a Lagrangian $\RP^2$ and $S$ a symplectic $(-1)$-sphere, such that
  $L$ and $S$ have trivial $\Z_2$-intersection, but
 $L$ and $S$ cannot be made disjoint with a smooth isotopy.

\item  For $k \geq 10$ there exists $L_{0}$ and $L_{1}$ that are $\Z_2$-homologous
smoothly embedded $\RP^2$'s, which are not smoothly isotopic.
Furthermore there are deformation equivalent symplectic forms $\w_{0}$ and $\w_{1}$
           on $M$ so that $L_i \subset (M, \w_{i})$ are Lagrangians.
\end{enumerate}
\end{prop}

Our symplectic packing results also provide ways to construct disjoint Lagrangian $S^{2}$'s and $\RP^{2}$'s,
and this leads to Floer-theoretic statements about properties of Lagrangians,
in particular showing that they are not superheavy with respect to the fundamental class of quantum homology \cite{EP09}.
Let $(X_{k}, \w_{0})$ be a del Pezzo surface, i.e.\ a monotone
$\CP^{2}\#k \overline{\CP}^{2}$ with $0 \leq k \leq 8$. 

\begin{thm}\label{t:RPNS}
    Any Lagrangian $\RP^{2}\subset(X_{k}, \w_{0})$ is not
    superheavy with respect to the fundamental quantum homology class
    $\ide \in QH_{4}(X_{k}, \w_{0})$ when $k \geq 2$.  Likewise for Lagrangian spheres $S^{2} \subset X_{k}$ when
    $k\geq 3$.
\end{thm}
\noindent
In Proposition~\ref{p:pafRP} we build Lagrangian $\RP^{2}$'s in all $X_k$.
We also note that this has an interesting application to detecting non-displaceable toric fibers
in certain symplectic surfaces with ``semi-toric type" structures (see Section \ref{s:final remarks}).

The structure of the paper goes as follows.  In Section \ref{s:CBP} we recall the necessary tools from relative
symplectic packing as well as the symplectic cutting procedure that lets us switch between Lagrangians and divisors.  This allows us to prove Theorem~\ref{t:BPRL-sphere}.
In Section \ref{s:-4 spheres} we study symplectic $(-4)$-spheres,
which is related to Lagrangian $\RP^2$'s via symplectic cuts.  In Section \ref{s:Lag RP2} we provide proofs of
Theorem \ref{t:TTRP2} and Proposition \ref{p:KRP2}.  In Section \ref{s:del Pezzo} we prove Theorem~\ref{t:RPNS} and discuss the relation between symplectic packing and Floer-theoretic properties.

\subsubsection*{Acknowledgements}
The authors warmly thank Selman Akbulut, Josef Dorfmeister, Ronald Fintushel, Robert Gompf, Dusa McDuff,
and Leonid Polterovich for their interest in this work and many helpful correspondences.  Particular thanks is due to Dusa McDuff for generously sharing early versions of her paper \cite{Mc12} with us, which plays a key role in our arguments.  We would also like to thank the anonymous referee for various valuable comments, suggestions, clarifications,
and pointing us to the fact that Corollary~\ref{c:TTS2}(1) leads to a description of the Hamiltonian isotopy classes
of Lagrangian spheres in a compact symplectic manifold where there are Hamiltonian knotted Lagrangian spheres.

\section{Connectedness of ball packings}\label{s:CBP}

Given a symplectic manifold $(M^{2n}, \w)$
the space of \textit{symplectic ball packings of $M$ of size $\bar{\lambda} = (\lambda_1,\dots,\lambda_k)$} is the space of smooth embeddings
$$
    E_{\bar\lambda}(M,\w) :=\left\{\phi:\coprod_{i=1}^k B^{2n}(\lambda_i) \to M : \phi^{*}\w = \w_{std} \mbox{ and $\phi$ is injective}\right\}
$$
where $B^{2n}(\l) = \{z \in \C^{n} : \pi \abs{z}^{2} \leq \l\}$ and $\w_{std} = dx \wedge dy$.
In \cite{Mc91, Mc93} McDuff established the connection between symplectic packing and the symplectic blow-up.  Given a symplectic packing $\phi \in E_{\bar{\lambda}}(M, \w)$ performing the symplectic blow-up results in the symplectic form $\w_{\phi}$ on $M\# k\overline{\CP}^{n}$
where $[\w_{\phi}] = [\pi^{*}\w] - \sum_{i=1}^{k} \lambda_{i}\mbox{PD}(E_{i})$ where $\pi: M\# k\overline{\CP}^{n} \to M$
is the blow-down map and $\mbox{PD}(E_{i})$ is the Poincare dual to the exceptional class $E_{i}$ corresponding to the blow-up of $M$ at the $i$-th ball.  See \cite{Mc93, MP94} for more details.

Lalonde and McDuff in \cite{La94,LM2, Mc98} developed a method known as inflation which builds a symplectic deformation of $(M, \w)$ by adding a two-form dual to a symplectic submanifold $C^{2} \subset (M^{4}, \w)$.  In particular McDuff's \cite[Theorem 1.2]{Mc98}, and its generalization by Li--Liu \cite[Proposition 4.11]{LL01}, proves if $(M^{4}, \w)$ is a closed symplectic manifold with $b_{2}^{+} = 1$, then any deformation between cohomologous symplectic forms is homotopic with fixed endpoints to an isotopy of symplectic forms. This homotopy is done by inflation along one parameter families of embedded holomorphic curves, whose existence is given by Taubes--Seiberg--Witten theory \cite{Ta96, Mc97}.  Using the relation between symplectic packings and symplectic forms on the blow-up,
McDuff \cite[Corollary 1.5]{Mc98} used that deformation implies isotopy to prove the space of symplectic packings is connected when $b_{2}^{+}(M) = 1$.

\subsection{In the complement of a symplectic submanifold}

Consider now the relative setting, where $Z \subset (M^{4}, \w)$ is
a closed embedded symplectic surface and one considers symplectic
packings $E_{\bar{\lambda}}(M\backslash Z, \w)$ in the complement of
$Z$. Biran \cite[Lemma 2.1.A]{Bi99} worked out how to inflate a
symplectic form along certain symplectic surfaces $C \subset M$ that
intersects $Z$ positively so that $Z$ stays a symplectic submanifold
through the deformation, so extending the connectedness of ball
packing just requires one to find the appropriate holomorphic curves
to inflate along.  When the Seiberg--Witten degree of $Z$ is
non-negative $d(Z) = c_{1}(Z) + Z^{2} \geq 0$, then $Z$ will be a
$J$-holomorphic curve for regular almost complex structures on $M$
so McDuff's argument in \cite[Corollary 1.5]{Mc98} generalizes
immediately.  

The situation is more delicate when $d(Z) < 0$, but
recent work of McDuff \cite{Mc12} builds the appropriate curves in
certain cases and leads to the following proposition, which is
implicit in \cite{Mc12}.  We have followed \cite[Corollary 1.5]{Mc98}
where McDuff proves the connectedness of ball packings in the absolute case when $b_{2}^{+}(M) = 1$.
Note that two packings $\phi_{0}, \phi_{1} \in E_{\bar{\l}}(M, \w)$ are connected if and only if
there is a symplectomorphism $F \in \Symp_{0}(M, \w)$ in the identity component of $\Symp(M, \w)$
such that $F\circ \phi_{0} = \phi_{1}$.

\begin{prop}\label{p:BPRS}
    Let $(W^{4}, \w)$ be a closed rational or ruled symplectic $4$-manifold and let
    $Z \subset W$ be a closed symplectic sphere, then
    the space of symplectic packing $E_{\bar\lambda}(W\backslash Z, \w)$ is connected.
\end{prop}

\begin{proof}
    Given two packings
    $\phi_{0}, \phi_{1}: \coprod_{i=1}^k B^{4}(\lambda_i) \to (W\backslash Z, \w)$
    by applying an element of $\Symp_{0}^{c}(W \backslash Z, \w)$, the 
    identity component of the group of compactly supported symplectomorphisms
    $\Symp^{c}(W \backslash Z, \w)$, we may assume that $\phi_{0} = \phi_{1}$ as maps
    when restricted to $\coprod_{i=1}^{k} B^{4}(a\lambda_{i})$ for $a > 0$ sufficiently small.
    
    Let $\bar{\w}_{0}$ and
    $\bar{\w}_{1}$ be the
    symplectic forms on the blow-up
    $\overline{W} = W \# k \overline{\CP}^{2}$
    associated to the ball packings $\phi_{0}$ and $\phi_{1}$.
    Pick a deformation of symplectic forms $\bar{\w}_{t}$
    on $\overline{W}$ as in the proof of \cite[Corollary 1.5]{Mc98}
    such that $\bar{\w}_{t}$ is constant in a neighborhood of $Z$ and in
    $H^{2}(\overline{W}; \R)$
    $$
        [\bar{\w}_{t}] = [\pi^{*}\w] - \sum_{i=1}^{d} (\lambda_{i}-\rho_i(t))\mbox{PD}(E_{i})
    $$
    where $\pi: \overline{W} \to W$ is the natural blow-down map, $\mbox{PD}(E_{i}) \in H^{2}(\overline{W}; \Z)$ is Poincare dual to the exceptional class $E_{i}$ associated to the exceptional
    divisor $e_{i} \subset \overline{W}$, and $\rho_i: [0, 1] \to [0, \lambda_{i})$ are smooth functions
    equal to $0$ in a neighborhood of $t=0,1$.
    By \cite[Proposition 1.2.9]{Mc12} there is a compactly supported isotopy $\{\bar{F}_{t}\}_{t}$
    in $\Diff_{0}(\overline{W}\backslash Z)$ such that $\bar{F}_{0} = \mbox{id}$ and
    $\bar{F}_{1}^{*}\bar{\w}_{1} = \bar{\w}_{0}$.  
    
    From here the proof proceeds exactly as in \cite[Corollary 1.5]{Mc98}.
    By blowing down $\bar{F}_{1}$ induces a symplectomorphism $F \in \Symp^{c}(W\backslash Z, \w)$ so that
    $F \circ \phi_{0} = \phi_{1}$, so to finish the proof it suffices to show
    $F \in \Symp_{0}^{c}(W\backslash Z, \w)$.  Note
    that the construction of $F = F^{(1)}$
    can be done in
    a family $F^{(a)} \in \mbox{Symp}^{c}(W\backslash Z, \w)$
    by starting the construction with respect to $\phi_{0}$ and $\phi_{1}$ being restricted to the domain
    $\coprod_{i=1}^{k} B^{4}(a\lambda_{i})$ for $a \in (0, 1]$.
    Since we assumed $\phi_{0}$ and $\phi_{1}$ are equal on sufficiently small balls,
    $F^{(a)} = \id$ for $a$ close to zero and hence
     $F \in \Symp_{0}^{c}(W\backslash Z, \w)$.
\end{proof}


\subsection{In the complement of a spherical Lagrangian}\label{s:SCBP}

In this subsection we will prove Theorem \ref{t:BPRL-sphere}.
The simple but key observation is to translate to a relative
setting with a symplectic sphere, which allows us to apply Proposition~\ref{p:BPRS}.
This translation will be done using the following construction.

\subsubsection{The symplectic cutting construction}

Let $(M^{2n}, \w)$ be a symplectic manifold and let $L \subset M$ be
a closed Lagrangian admitting a metric with periodic geodesic flow.
Let $\cN$ be a Weinstein neighborhood of
$L \subset M$, then by using the periodic geodesic flow we can
perform a symplectic cut \cite{Le95} along $\partial\cN$.  This cut
creates a pair of new symplectic manifolds $W^{+}_{\cN}$ and $W^{-}_{\cN}$ each having
$Z_{\cN}$ as a codimension $2$ symplectic submanifold, where the Euler classes of $Z_{\cN}$'s normal bundles in $W^{\pm}_{\cN}$ satisfy $e(\nu^{+}) = -e(\nu^{-}) \in H^{2}(Z_{\cN})$.  The manifolds are
$$
\mbox{$W^{+}_{\cN} = (M\backslash \Int\cN)/\sim$ \quad $W^{-}_{\cN} = \cN/\sim$
\quad $Z_{\cN} = \d\cN/\sim$}
$$
where the equivalence relations are given by identifying point on $\d\cN$ that lie on the same geodesic,
and the symplectic manifold $M$ can be recovered by taking the symplectic sum \cite{Go95}
of $W^{\pm}_{\cN}$ along $Z_{\cN}$
\begin{equation}\label{e:sum}
    M = W^{+}_{\cN}\, \#_{Z_{\cN}} W^{-}_{\cN}.
\end{equation}

\begin{remark}\label{r:names}
Let us record what arises when $L$ is a Lagrangian $S^{2}$ or $\RP^{2}$.
\begin{enumerate}[(1)]
\item Lagrangian $S^{2}$: We have $W^{-} = (S^{2} \times S^{2}, \sigma \oplus \sigma)$ with
$L \subset W^{-}$ being the anti-diagonal Lagrangian sphere, where $Z \subset W^{-}$ is the diagonal symplectic sphere and $Z \subset W^{+}$ is a $(-2)$ symplectic sphere.
\item Lagrangian $\RP^{2}$: We have $W^{-} = \CP^{2}$ with $L \subset W^{-}$ being the the standard
$\RP^{2}$, where $Z \subset W^{-}$ is the quadric $Q = \{z_{0}^{2} + z_{1}^{2} + z_{2}^{2} = 0\}$ and
$Z \subset W^{+}$ is a $(-4)$ symplectic sphere.
\end{enumerate}
For more details and other examples see \cite{Au07}.
\end{remark}

In \cite[Theorem 1.1]{HK99} Hausmann--Knutson determined the effect of symplectic cutting in general on the rational cohomology ring and it leads to the following lemma.  Note that \cite[Theorem 1.1]{HK99} assumes the symplectic cut is by a global Hamiltonian $S^{1}$-action, which need not be true in our case.
However since $H^{*}(M; \Q) \to H^{*}(\cN; \Q)$ is surjective if $\cN \subset M$ is a Weinstein neighborhood for a Lagrangian $S^{2}$ or $\RP^{2}$, the proof of \cite[Theorem 1.1]{HK99} still applies in our case.

\begin{lem}\label{l:b2+}
    Let $(M^{4}, \w)$ be a closed symplectic manifold with $L \subset M$ a Lagrangian
    $S^{2}$ or $\RP^{2}$ and let $W^{+}_{\cN}$ be the symplectic manifold built by cutting out
    a Weinstein neighborhood $\cN$ of $L$, then we have the following:
    \begin{enumerate}[(1)]
    \item $b_{2}^{+}(W^{+}_{\cN}) = b_{2}^{+}(M)$.
    \item If $L = S^{2}$, then $b_{2}^{-}(W^{+}_{\cN}) = b_{2}^{-}(M)$.  If $L = \RP^{2}$, then
    $b_{2}^{-}(W^{+}_{\cN}) = b_{2}^{-}(M) + 1$.
    \item If $M$ is rational or ruled, then $W^{+}_{\cN}$ is rational or ruled respectively.
\end{enumerate}
\end{lem}
\begin{proof}
    For (1) and (2):
    If $L$ is a Lagrangian $S^{2}$, then $H^{*}(W^{+}_{\cN}; Q) \cong H^{*}(M; \Q)$ as rings by \cite{HK99}.
       If $L$ is a Lagrangian $\RP^{2}$,
       then by \cite{HK99}
       the intersection forms are related by
    \begin{equation}\label{e:Q}
        Q_{W^{+}_{\cN}} = Q_{\overline{\CP}^{2}} \oplus Q_{M} = (-1) \oplus Q_{M}
    \end{equation}
    where
    $PD(Z_{\cN}) \in H^{2}(W^{+}_{\cN};\Q)$ is identified with
    $2h \in H^{2}(\overline{\CP}^{2};\Q)$.

        For (3): Let $\kappa$ denote the symplectic Kodaira dimension, then by
    \cite[Theorem 1.1]{Do10a} and \eqref{e:sum} if follows that $\kappa(W^{+}_{\cN}) \leq \kappa(M)$.
    For closed $4$-dimensional symplectic manifolds having
    $\kappa = -\infty$ is equivalent to being rational or ruled
    \cite{Liu96}.
    Since \cite{HK99} gives $H_{1}(W^{+}_{\cN}; \Q) \cong H_{1}(M;\Q)$ it follows that
    $W^{+}_{\cN}$ is rational if $M$ is rational, and likewise for ruled.
\end{proof}
\noindent
In the case of a Lagrangian sphere, Lemma~\ref{l:b2+} can also be proved by noting that $W^{+}_{\cN}$ is a symplectic deformation of $M$.

\subsubsection{Proving Theorem \ref{t:BPRL-sphere}}

In the setting of the symplectic cutting construction, by design $W^{+}_{\cN}\backslash Z_{\cN}$ is symplectomorphic to $M\backslash \cN$.  Therefore we immediately have the following observation on symplectic packings of $M\backslash L$.

\begin{lem}\label{l:RLRS}
    The space of ball packings
    $E_{\bar\lambda}(M\backslash L)$ in $M\backslash L$
    is connected if and only if
    the space of ball packings
    $E_{\bar\lambda}(W^{+}_{\cN} \backslash Z_{\cN})$ in $W^{+}_{\cN} \backslash Z_{\cN}$
    is connected for all sufficiently small Weinstein neighborhoods
    $\cN$ of $L \subset M$.
\end{lem}
\noindent
This in turn leads to the proof of Theorem~\ref{t:BPRL-sphere}.
\begin{proof} [Proof of Theorem~\ref{t:BPRL-sphere}]
    Let $(W^{+}_{\cN}, \w)$ be the result of cutting out a Weinstein neighborhood $\cN$ of $L$ in $M$,
    with $Z_{\cN} \subset W^{+}_{\cN}$ being the resulting symplectic sphere.
    By Lemma~\ref{l:b2+} we have that
    $W^{+}_{\cN}$ is rational or ruled,
    so by Proposition~\ref{p:BPRS}
    the space of ball packings in $W^{+}_{\cN}\backslash Z_{\cN}$ is connected.
    The theorem now follows from Lemma~\ref{l:RLRS}.
\end{proof}



\section{Symplectic $(-4)$-spheres in rational manifolds}\label{s:-4 spheres}

In this section we provide a classification of a special type of
classes which can be represented by symplectic $(-4)$-spheres in
symplectic rational manifold $(W^{4}, \w)$.  This will be crucial to our study of Lagrangian
$\RP^2$'s due to the symplectic cutting construction.

We first briefly establish some notation.  If $M = S^{2} \times S^{2}$, then we will let $A,B \in H_{2}(M;\Z)$ be the homology classes for each factor.
If $M = \CP^{2}\#k\overline{\CP}^{2}$, we will represent its second homology classes by the
basis $\{H, E_1,\dots,E_k\}$, where $H$ is the line class and $E_i$
are orthogonal exceptional classes.  Denote the class 
$$K_{0} = -3H+E_1+\cdots+E_k$$ and by
\cite[Theorem 1]{LL01} we may always assume that this is the canonical class
for symplectic rational manifolds. Define
the \textit{$K_0$-exceptional classes} as the spherical classes $C$ satisfying $K_0\cdot C=-1$,
and by \cite{LL01} these classes must be represented by an symplectic exceptional sphere.

There is a special kind of transformation on the second homology group of rational manifolds called the \textit{Cremona
transform}, which are reflections
$$
    A \mapsto A + (A\cdot L_{ijk})\, L_{ijk} \quad\mbox{where}\quad L_{ijk} = H-E_i-E_j-E_k \mbox{ for $i>j>k$}
$$
that preserve the class $K_{0}$.  When $k\leq 3$, we also include the reflection with respect to $L_{ij}=E_i-E_j$.
Cremona transformations can
always be realized by diffeomorphisms \cite{Li02} (explicitly this can be done by a smooth version of Seidel's
Dehn twist).  See \cite{MSc12, LW12} for
more detailed discussions on these transformations.  If two classes are connected by a series of Cremona transforms,
we say they are \textit{Cremona equivalent}.


Let $Z\subset W$ be a symplectic $(-4)$-sphere, then the
\textit{rational blow-down along} $Z$ is the symplectic manifold built by
performing a symplectic sum \cite{Go95} of $W$ along $Z$ with the standard quadric $Q$ in $\CP^2$
and is denoted $M = W_{Z}\#_Q\CP^2$.  This operation is the inverse of the symplectic cutting
construction from Section~\ref{s:SCBP} in the case of a Lagrangian $\RP^{2} \subset M$.

\brmk
    Unlike the symplectic cutting operation that keeps the manifold rational or ruled in our case by
    Lemma~\ref{l:b2+}, the same is not true when we blow-down $(-4)$ symplectic spheres.
    For example the class $2K_{0} = 6H-2(E_1+\dots+E_{10})$
    is represented by an embedded $(-4)$-symplectic sphere
    for an appropriate symplectic form on $\CP^{2}\# 10\overline{\CP}^{2}$,
        but the symplectic blow-down is not a rational manifold \cite[Section 4.2]{Do10b}.
\ermk

\begin{prop}\label{p:C-4C}
      Let $(W^{4}, \w)$ be a symplectic rational manifold
      and $Z\subset W$ be an embedded symplectic
      $(-4)$-sphere.  If $W = S^{2} \times S^{2}$, then $[Z] = A-2B$ or $B-2A$.
      If $W = \CP^{2}\#k\overline{\CP}^{2}$, then $k\geq 2$ and furthermore $[Z] \in H_{2}(W, \Z)$
      is Cremona equivalent to $-H + 2E_{1} - E_{2}$ provided
      the rational blow-down $M$
      of $W$ along $Z$ is a rational manifold.
\eprop

\begin{proof}
     A direct computation with the adjunction formula implies all but
     the case where $W = \CP^{2} \# k \overline{\CP}^{2}$ with $k \geq 4$.
     From the relation between rational blow-down and symplectic cutting
     from Section~\ref{s:SCBP},  it follows from equality \eqref{e:Q} that
    $b^-_{2}(W) = 1+ b^-_{2}(M)$.
    It follows that $M$ is not minimal since it is rational and
    $b^-_{2}(M) = k-1 \geq 3$.

By \cite[Theorem 1.2]{Do10b} it follows that either
          \begin{enumerate}[(i)]
             \item there exists a symplectic exceptional sphere disjoint
             from $Z$, or
             \item there exist two disjoint symplectic exceptional spheres
             $C_{1}$ and $C_{2}$
             each intersecting $Z$ exactly once and positively.
          \end{enumerate}
     Therefore we can assume after blowing down some number of exceptional
     spheres disjoint from $Z$ that the resulting manifold $\overline{W}$
     is minimal or (ii) occurs for $Z \subset \overline{W}$.

     If we reach a minimal manifold then $\overline{W}=S^2\times S^2$, since the other possibility,
     $\CP^{2}$, does not have a symplectic $(-4)$-curve.
     So $[Z]$ is either $A-2B$ or $B-2A$ in $H_{2}(\overline{W};\Z)$ and
     since we needed to have blow-down at least $3$ exceptional spheres in $W$ to reach $\overline{W}$,
     by blowing back up to $W$ it follows that the classes $A-2B$ and $B-2A$ can be
     represented in the form of part (ii) in our assertion, by an appropriate change of basis from classes $A,B$
     to the standard basis consisting of $\{H, E_1,\dots, E_k\}$.

    Assume now (ii) occurs for $Z \subset W$, where $W$ has no exceptional spheres disjoint from
    $Z$.
    From the classification of exceptional classes
    in rational manifolds up to Cremona equivalence \cite[Proposition 1.2.12]{MSc12},
    \cite[Corollary 4.5]{LW12}, any exceptional class is Cremona equivalent to $E_i$ for some $i$.
    Also since Cremona transformations can be realized by diffeomorphisms,
    we may now assume without loss of generality $[C_i]=E_i$ for $i=1,2$.

    Taking our divisor $Z$ into account, the effect of blowing down both $C_{i}$ can be described as
    performing symplectic fiber sums \cite[Theorem 1.4]{Go95}
    of $(W, Z)$ with $(\CP^{2}, \ell)$ along
    $C_{i}$ and $\ell_{i}$, where $\ell$ is a line intersecting the line $\ell_{i}$ transversally,
    for $i=1$ and $2$.
    The result is $(W', Z')$ where $Z'$ is a symplectic $(-2)$-sphere in a rational
    symplectic manifold $W'$ and under $\iota: H_2(W')\rightarrow H_2(W)$
    the canonical inclusion we have $\iota[Z']=[Z]-E_1-E_2$.

    It follows from \cite[Proposition 4.10]{LW12} that $[Z'] \in H_2(W')$ is Cremona equivalent to the class
    $E_3-E_4$.  Since Cremona transforms are given by reflections about homology classes,
    the inclusion $\i$ translates the Cremona equivalent of $[Z']$ and $E_{3} - E_{4}$ in 
    $H_{2}(W')$
    to a series of Cremona transforms of $H_{2}(W)$ that fix $E_{1}$ and $E_{2}$
    while sending $\iota[Z']$ to the class $E_3-E_4-E_1-E_2$.  This class
    is equivalent
    to $-H+2E_3-E_4$ by reflection with respect to $H-E_1-E_2-E_3$,
    and $-H+2E_3-E_4$ is equivalent to
    the form in the assertion.
\end{proof}

\begin{cor}\label{c:EnoughE}
    If $(W, \w)$ is a symplectic rational manifold and
    $Z$ is an embedded symplectic
    $(-4)$-sphere, then in the complement of $Z$
    there are $n = b_{2}^{-}(W)-1$ disjoint exceptional symplectic spheres
    $\{D_i\}_{i=1}^{n}$.  Moreover, given any set of orthogonal exceptional classes
    pairing trivially with $[Z]$, there are disjoint exceptional spheres in $W\backslash Z$ representing
    these classes.
 \end{cor}
\begin{proof}
    There is nothing to prove if $W = S^{2} \times S^{2}$ or $\CP^2\#\overline{\CP}^{2}$, so
    by Proposition~\ref{p:C-4C} we can assume $W = \CP^{2} \# (n+1)\overline{\CP}^{2}$
    with $n \geq 1$ and $[Z] = -H + 2E_{1} - E_{2}$.
    In this case,
    $$
        \mbox{$E_{1}' = H-E_{1}-E_{2}$ \quad and \quad$E_{i}'= E_{i+1}$ for $i=2,\dots,n$.}
    $$
    is a set of $n$ orthogonal exceptional classes that pair trivially with $[Z]$.
    Now it suffices to prove the second part, which follows from
    McDuff's result \cite[Proposition 1.2.5]{Mc12} that
      asserts we can find a compatible  almost complex structure $J$ on $(W, \w)$ for which
      $Z$ is a complex submanifold
    and the classes $E_{i}'$ are represented by a $J$-holomorphic embedded spheres $D_{i}$.
    By positivity of intersections is follows that the $D_{i}$ are disjoint from each other and $Z$.
\end{proof}

\begin{defn}\label{d:push off}
If $(W^{4}, \w)$ is a symplectic rational manifold and $Z \subset W$ is a symplectic $(-4)$-sphere, then
a set of $n = b_{2}^{-}(W) - 1$ disjoint exceptional spheres $\{D_i\}_{i=1}^{n}$ in $W \backslash Z$
is called a \emph{push off set of $Z$} and the set $\{[D_i]\}_{i=1}^{n}$ of exceptional classes is called a
\emph{push off system of $Z$}.
\end{defn}

\section{Lagrangian $\RP^2$'s in a rational manifolds}\label{s:Lag RP2}

In this section $(M^{4}, \w)$ will always be a symplectic $\CP^{2}\# k \overline{\CP}^{2}$
where the canonical class is $K_{0} = -3H + E_{1} + \cdots + E_{k}$
and $L \subset M$ will be a Lagrangian $\RP^{2}$. Note that
$S^2\times S^2$ does not contain Lagrangian $\RP^2$ since it does not
have a $\Z_2$-homology class with non-trivial self-pairing, so we
need not put it into consideration.
We will let $(W, Z) = (W^{+}_{\cN}, Z_{\cN})$ be the result of performing the symplectic cutting procedure from Section~\ref{s:SCBP} on $(M, L)$.  Note by construction that $Z \subset W$ is a symplectic
$(-4)$ sphere and $W$ is a symplectic rational manifold $\CP^2\#(k+1)\overline{\CP}^2$ by Lemma~\ref{l:b2+}.

\subsection{Push off and orthogonal systems}

Translating our knowledge about symplectic $(-4)$ spheres leads to the following result.

\begin{prop}\label{p:DfRP}
     If $(M^{4}, \w)$ is a symplectic rational manifold with $L \subset M$ a Lagrangian $\RP^{2}$,
     then there is a set of $k = b_{2}^{-}(M)$ disjoint exceptional spheres
     $\{D_i\}_{i=1}^{k}$ in $M \backslash L$.
\end{prop}
\begin{proof}
    By construction $W\backslash Z$ is symplectomorphic to a subset of $M \backslash L$
    and by Corollary~\ref{c:EnoughE} there are $k = b_{2}^{-}(W) - 1 = b_{2}^{-}(M)$
    disjoint exceptional spheres $\{D_i\}_{i=1}^{k}$ in $W \backslash Z$.
\end{proof}

\noindent
In light of Proposition \ref{p:DfRP}, we introduce the following analogue of Definition \ref{d:push off}.

\begin{defn}\label{d:POLAG}
    If $(M^{4}, \w)$ is a symplectic rational manifold and $L \subset M$ is a Lagrangian $\RP^{2}$,
    then a set $k = b_{2}^{-}(M)$ disjoint exceptional spheres $\{D_i\}_{i=1}^{k}$ in $M \setminus L$
    is called a \textit{push off set} of $L$ and the set
    $\{[D_i]\}_{i=1}^{k}$ of exceptional classes is called a \textit{push off system} of $L$.
\end{defn}

The following proposition shows the importance of the notion of a push off system in proving Theorem~\ref{t:TTRP2}.

\begin{prop}\label{p:isotopy criterion}
Suppose two Lagrangian $\RP^2$'s in a symplectic rational manifold have a common push off system,
then they are both (1) smoothly isotopic and (2) Torelli symplectomorphic.
\end{prop}

\bpf
    For the first claim,
    Corollary~\ref{c:EnoughE} and \cite[Proposition 3.4]{LW12} ensures that up to symplectic isotopy the Lagrangian
    $L_{i}$ share common a push off set.  By blowing down the push-off set we have
    Lagrangian $\RP^{2}$'s in $\CP^{2}$, which are unique up to Hamiltonian isotopy \cite{Hi12, LW12}.
    This can be lifted to a smooth isotopy of $M$ since avoiding a blow-up region along an isotopy is equivalent to avoiding
    a point in the smooth category.  See also the argument in the proof of \cite[Theorem 1.6]{LW12}.
        The proof of the second claim is identical to the proof of Corollary \ref{c:TTS2} in Appendix~\ref{s:A2}.
\epf

The rest of the section lists some homological restrictions on Lagrangian $\RP^2$ classes, which will prove useful.  In particular, the following definition, albeit purely topological, will provide more accurate homological information regarding a Lagrangian $\RP^2$ class.

\begin{defn}\label{d:mod2 push off}
Let $(M, \w)$ be a symplectic $\CP^{2}\#k\overline{\CP}^{2}$ and $A\in H_2(M;\Z_2)$,
then a \emph{$\Z_2$-orthogonal system} for $A$ is a set
$\{F_{i}\}_{i=1}^{k}$ of pairwise orthogonal exceptional classes in $H_{2}(M;\Z)$ such that $A\cdot F_i=0$
in $\Z_{2}$-homology for all $i$.  If $A$ admits a $\Z_2$-orthogonal system, then $A$ is called a
\emph{$K_{0}$-Lagrangian $\RP^{2}$ class}.
\end{defn}

\begin{remark}\label{r:POvO}
    Note that a $\Z_{2}$-orthogonal system consists of only homological conditions for the
    exceptional classes, while a push off system (Definition \ref{d:POLAG}) requires exceptional representatives disjoint from $L$.
    We have that
    $$
        \{\text{push of systems of }L\} \subset \{\Z_2\text{-orthogonal systems of }[L]\}
    $$
    as collections of subsets of $H_{2}(M;\Z)$.
\end{remark}

\blem\label  {l:RP2 classes} Let $(M, \w)$ be a symplectic $\CP^{2}\# k \overline{\CP}^{2}$.
\begin{enumerate}[(1)]
\item If $L \subset M$ is a Lagrangian $\RP^{2}$, then $[L] \in H_{2}(M; \Z_{2})$ is a $K_0$-Lagrangian
$\RP^2$ class.

\item Each $K_0$-Lagrangian $\RP^2$ class is Cremona equivalent to the $\Z_2$-reduction
$H$.

\item The $\Z_{2}$-reduction $H \in H_{2}(M; \Z_{2})$ is the unique $K_0$-Lagrangian
$\RP^2$ class when $k \leq 2$.

\item Every $K_0$-Lagrangian $\RP^2$ class
is represented by a smoothly embedded $\RP^2$, which is Lagrangian for some symplectic form on $M$.
\end{enumerate}
\elem

\begin{proof}
Part (1) follows directly from Proposition \ref{p:DfRP} and definitions.

To prove (2), we need to show the following statement: given $k$ orthogonal exceptional classes,
there is a Cremona transform sending this set to $\{E_1,\dots, E_k\}$.  This is proved in
 \cite[Corollary 4.5]{LW12} for a set of at most $k-2$ exceptional classes (and does not
 hold in general for a set of $k-1$ exceptional classes, e.g.\ when $k=2$ and the set is $\{H-E_1-E_2\}$).

Apply the proved case for $k-2$ exceptional classes to $\{E_3,\cdots, E_k\}$.  This reduces the problem to the case of
$\CP^2\#2\overline{\CP}^2$.  But the only orthogonal set of exceptional classes containing 2 elements is $\{E_1, E_2\}$.
This means the induction necessarily terminates in this case and the resulting Cremona transform sends the
 set of exceptional classes to $\{E_i\}_{i=1}^{b^-}$ as desired.
 It is then clear that the only non-trivial $\Z_2$-class orthogonal to this set is the $\Z_2$-reduction of $H$.

For part (3), when $k\leq 2$ there is a unique set of $k$ orthogonal exceptional classes, so the only
$K_0$-Lagrangian $\RP^2$ class is the reduction of $H$.

It is easy to make the proper transform of the standard $\RP^2\subset \CP^2$ Lagrangian, just by packing sufficiently small balls in the complement.  Therefore part (4) follows from (2) since Cremona transforms are homological actions of smooth Dehn twists.
\end{proof}

Return now to the case where $L \subset M = \CP^{2}\# k \overline{\CP}^{2}$ is a Lagrangian $\RP^{2}$
and $(W, Z)$ is the result of performing the symplectic cutting procedure.
The next lemma describes the relationship between the push off systems of $L \subset M$ and
$Z \subset W$.

\blem\label{l:criterion}
The isomorphism $H_{2}(W \backslash Z; \Z) \to H_{2}(M \backslash L; \Z)$ induces a bijection:
$$\tau:\{\text{push off systems of }Z\}\rightarrow\{\text{push off systems of }L\}$$
\elem
\bpf
The map $\tau$ is defined as follows, a push off set of $Z$ survives the rational blow-down along $Z$
and hence defines a push off set of $L$.  The homology isomorphism ensures $\tau$ is well-defined.
Since the map $H_{2}(M\backslash L; \Z) \to
H_{2}(M; \Z)$ is injective by the long exact sequence for the pair $(M, M\backslash L)$,
and there is an isomorphism $H_{2}(W \backslash Z; \Z) \cong H_{2}(M \backslash L; \Z)$,
it follows that $\tau$ is injective.

For surjectivity we argue as follows. Given a push off system of $L$, perform a symplectic cut on a sufficiently small
 Weinstein neighborhood $\cN'$ of $L$ to get a push off system
for $Z'$ in an ambient symplectic manifold $W'$.  Since $(W,Z)$ and $(W', Z')$ differ by a symplectic deformation near the divisors and exceptional classes are invariant under symplectic deformation,
Corollary~\ref{c:EnoughE} builds a push off set for $Z$ that $\tau$ maps to the original push off system of
$L$.
\epf

\subsubsection{Proof of Theorem~\ref{t:TTRP2}}

\begin{lem}\label{l:counting}   Suppose $L$ is a Lagrangian $\RP^2$ in $M = \CP^2\#k\overline{\CP}^2$
for $k\leq 8$, then each $\Z_2$-orthogonal system of $[L]$ is a push off system of $L$.
\end{lem}

\begin{proof}
    Note there is nothing to prove if $k = 0$, and so if we let
    $(W, Z)$ be the result of doing the symplectic cut to $(M, L)$, by Lemma~\ref{l:b2+}
    we may assume that $W = \CP^{2}\#(k+1)\overline{\CP}^{2}$ where $1\leq k \leq 8$.

    When $k \leq 8$, there are only finitely many exceptional classes
    and hence a finite number of push off systems and $\Z_2$-orthogonal systems.
    Hence by the inclusion in Remark~\ref{r:POvO} it suffices to perform a count
    of each of these sets.  As we will now explain this reduces to a count
    of certain homology classes, and this count is performed in Appendix B.

    By Lemma~\ref{l:RP2 classes} we can assume that $[L] = H \in H_{2}(M; \Z_{2})$,
    so we only need to count $\Z_{2}$-orthogonal systems for $H \in  H_{2}(M; \Z_{2})$.
    This is equivalent to counting sets of $k$ orthogonal
    exceptional classes in $H_{2}(M; \Z)$ that are $\Z_{2}$-orthogonal to $H$.

    By Lemma~\ref{l:criterion} push off systems of $L$ in $M$ correspond to
    push off systems of $Z$ in $W$ and by Proposition~\ref{p:C-4C} we
    can assume that $[Z] = S = -H + 2E_{1} - E_{2}$. So by the second part of
    Corollary \ref{c:EnoughE}
    it suffices to count sets of $k$ orthogonal exceptional classes in $H_{2}(W;\Z)$ that
    pair trivially with $S$ in $H_{2}(W; \Z)$.
\end{proof}

We can now prove Theorem~\ref{t:TTRP2}.

\begin{proof}[Proof of Theorem \ref{t:TTRP2}]
    If $L_{1}$ and $L_{2}$ are in the same $\Z_{2}$-homology class, then they
    have a common $\Z_{2}$-orthogonal system and hence by Lemma~\ref{l:counting}
    they share a common push off system since $k \leq 8$.
    Therefore by Proposition \ref{p:isotopy criterion} they are smoothly isotopic
    and Torelli symplectomorphic.
\end{proof}


\subsection{Homology correspondence}

We now attempt to understand the relation between $\Z_2$-orthogonal classes and push-off systems of
Lagrangian $\RP^2$ in more detail.  This leads to the failure of Lemma \ref{l:counting} when $k\geq9$
thus a construction for twisted Lagrangian $\RP^2$ as in Proposition \ref{p:KRP2}.

\subsubsection{Associated basis}

Here we will explicitly describe the homological effect of rational blow-down in
rational manifolds for $M = W_{Z}\#_{Q}\CP^{2}$.

\begin{lem}\label{l:basis}
   \begin{enumerate}[(1)]
    \item There is an orthogonal basis $\{H', E_{0}', E_{1}', \dots, E_{k}'\}$ of $H_{2}(W;\Z)$,
    where $H'$ is a line class and the $E_{i}'$ are exceptional classes.
    Elements of this basis are represented by disjoint symplectic spheres $\{h', e_{0}', \dots, e_{k}'\}$.
    The symplectic spheres can be picked so each of $e_{1}', e_{2}', e_{3}'$ intersect
    $Z$ exactly once and positively, while $h', e_{4}', \dots, e_{k}'$ are disjoint from $Z$.

    \item There is an orthogonal basis $\{H, U_{1}, U_{2}, U_{3}, E_4,\dots, E_{k}\}$ for $H_{2}(M;\Z)$,
    where $H$ is a line class and the rest are exceptional classes.
    Elements of this basis are
    represented by disjoint symplectic spheres $\{h, u_{1}, u_{2}, u_{3}, e_4, \dots, e_{k}\}$.
    Furthermore $[L] = U_{1} + U_{2} + U_{3} \in H_{2}(M;\Z_{2})$.
   \end{enumerate}
\end{lem}
\begin{proof}
    By Proposition~\ref{p:C-4C} we may assume that $[Z] =  E_{0}' - E_{1}' -E_{2}' -E_{3}'$,
    since it is Cremona equivalent to $-H' + 2E_{0}' - E_{1}'$, and from here
    \cite[Proposition 1.2.7]{Mc12}
    builds the appropriate spheres $h'$ and $e_i'$ for $i=1, \dots, k$
    for generic almost complex structures for which $Z$ is holomorphic.
   The symplectic representative of $e_0'$ will not be used later, but one can always attain
   a desired representative by perturbing the almost complex structure from the previous step to a
   generic one.  This proves part (1).

For part (2) since the curves $h', e_{4}', \dots, e_{k}'$ in $W$
are disjoint from $Z$ they can also be viewed as curves in $M= W_{Z}\#_{Q}\CP^{2}$ after
performing the symplectic sum.
So we are
left to build the curves $u_{1}, u_{2}, u_{3}$ in $M$.  To build
$u_{1}$ the gluing method for relative Gromov--Witten invariants
\cite{LR01, IP04} allow us to glue the curves $e_{1}'$ and $e_{2}'$
in $W$, which intersect $Z$ transversally, to a degree one curve
$e_{12}'$ in $\CP^{2}$ that intersects the quadric twice at the
points $x_{1}$ and $x_{2}$ where $x_{i} = e_{i}' \cap Z$. Denote the
resulting holomorphic curve $u_{1}$ and likewise build the
holomorphic curves $u_{2}$ and $u_{3}$.  By construction we have
that in $\Z_{2}$-homology $[L] \cap U_{i} = 1$ and $[L] \cap H = [L]
\cap E_{i} = 0$, and this is enough to determine that $[L] = U_{1} +
U_{2} + U_{3}$.
\end{proof}

 \subsubsection{The monomorphism $\iota$}

Let $\cL_L\subset H_2(M,\Z)$ be the subgroup formed by elements
$[\alpha]$ with trivial $\Z_{2}$-intersection product with $[L] \in
H_{2}(M;\Z_{2})$. In Lemma~\ref{l:HERB} we will give an explicit
monomorphism $\iota:\cL_L\hookrightarrow H_2(W;\Z)$, but first we
need the following lemma whose proof was pointed out to us by Robert
Gompf.

\begin{lem}
    Let $X^{4}$ be a closed oriented manifold with $L$ an embedded $\RP^{2}$.
    If a homology class $[\alpha] \in H_{2}(X;\Z)$ pairs trivially $[\alpha] \cdot [L] = 0$ with $[L]$
    in $\Z_{2}$-homology, then
    $[\alpha]$ has a representative cycle $\alpha$ such that $\alpha \cap L = \emptyset$.
\end{lem}
\begin{proof}
    Pick an orientable representative cycle $C$ of $[\alpha]$ intersecting $L$ transversely
    such that the number of total intersection points is even.
    One may cancel a pair of the intersections $x_{1}, x_{2}$ in the following way:
    By fixing a local orientation on the interior of the unique $2$-cell of $\RP^{2}$, one can
    define local intersection numbers $\pm 1$ at $x_{1}$ and $x_{2}$.
    Now pick a path $\gamma$ in $\RP^{2}$ connecting $x_{1}$ and $x_{2}$ that
    either reserves or preserves the local orientation, depending on if the local
    intersection numbers match or not. 
    Form a new cycle now by removing small neighborhoods of $x_{1}$ and $x_{2}$ in $C$ and then
    connecting the boundaries with a tube $S^{1} \times [0,1]$ corresponding to a small circle section
    over the path $\gamma$ in the normal bundle of $\RP^{2} \subset X$.    
    Since $X^{4}$ is orientable,
    flipping the local orientation in $\RP^2$ is equivalent to flipping
    the normal orientation and hence we have a new orientable cycle representing $[\alpha]$
    with two fewer intersections with $L$.
\end{proof}

\blem\label{l:HERB}  Suppose $M=\CP^2\#k\overline{\CP}^{2}$ with $k\geq 3$, then
 there is a monomorphism  $\iota:\cL_L\hookrightarrow
   H_2(W;\Z)$
   with the following properties:
   \begin{enumerate}[(1)]
    \item It preserves the $\Z$-intersection product.
\item    In the bases from Lemma \ref{l:basis} if
   $[\alpha]=aH-t_1U_{1}-t_2 U_{2}-t_3U_{3}$ 
   then
   \begin{equation}\label{e:iota}
   \iota([\alpha])=aH'- \tfrac{-t_1+t_2+t_3}{2}E_0' - \tfrac{t_1-t_2+t_3}{2}E_1'
   - \tfrac{t_1+t_2-t_3}{2}E_2'
   -\tfrac{t_1+t_2+t_3}{2}E_3'
   \end{equation}
   and $\iota(E_{i}) = E_{i}'$ for $i \geq 4$.
       \item If $E \in H_{2}(M;\Z)$ is an exceptional class, then
    $E$ has an embedded symplectic spherical representative disjoint from $L$ for some
    $K_0$-symplectic form if and only if
    $E \in \mathcal{L}_{L}$
        and $\iota(E)$ is also an exceptional class.
 \end{enumerate}
  \elem
\bpf
	Given $[\alpha] \in \cL_{L}$ define $\i([\alpha]) \in H_{2}(W;\Z)$ by picking a representative cycle
	$\alpha$ in $M\backslash L$ and view it as a cycle in $W$.  Since $H_{2}(M\backslash L) \to H_{2}(M)$
	is injective, it follows that this map is well-defined.  It is clear that
	this map preserves the intersection pairing.

Part (2) is verified by comparing intersection numbers with the
     associated basis of the cycle $\alpha$.
     If $[\alpha]=aH-t_1U_{1}-t_2U_{2}-t_3U_{3}$ 
     and $\iota([\alpha])=aH-\sum_{i\geq0} b_i'E_i'$, then $b_i=b_i'$, $b_0'=b_1'+b_2'+b_3'$,
     $b_2'+b_3'=t_1$, $b_1'+b_3'=t_2$, $b_1'+b_2'=t_3$ from construction of $u_1$, $u_2$, and $u_3$ in Lemma \ref{l:basis}.
     This is exactly the form given by the statement of the lemma after elementary computation.

     For part (3), the `only if' part is trivial.  We prove the `if' part, the definition of $\iota$ implies
that $\langle\iota(E), [Z]\rangle=0$, so by \cite[Theorem 1.2.1]{Mc12},
we have a representative of symplectic exceptional sphere in $W$ disjoint from $Z$.  This is also a symplectic exceptional sphere
in $M\backslash L$ that represents $E$.
\epf



\subsection{Proof of Proposition~\ref{p:KRP2}}
\subsubsection{A Lagrangian $\RP^2$ in $B_{3} = B\#3\overline{\CP}^{2}$}

\blem\label{l:CRP2}
     For a symplectic ball blown-up three times
     $(B_3=B\#3\overline{\CP}^2, \w)$ if the exceptional classes have
     equal small symplectic area, then $B_3$ admits $L$ an
     embedded Lagrangian $\RP^2$ such that $[L] = E_{1} + E_{2} + E_{3} \in H_{2}(B_{3}; \Z_{2})$.
\elem

\bpf
     Suppose the three blow-ups are of size $\alpha$, we are going to
     obtain $B_3$ in a different way as follows.
     We perform the following blow-up on $B$ instead:
     blow-up once of size $t=\frac{3\alpha}{2}+\epsilon$, where
     $0<\epsilon\ll \alpha$.  We denote this
     resulting exceptional sphere as $e_0$ and its class as $E_0$.
     Again perform three more blow-ups on $e_0$ with equal size
     $\frac{\alpha}{2}-\epsilon$, where the resulting exceptional spheres and
     classes are denoted as $e_i$ and $E_i$, respectively, for
     $i=1,2,3$.  We now have a symplectic sphere $e_{0123}$ in class $E_0-E_1-E_2-E_3$
     of self-intersection $(-4)$ and area $4\epsilon$ in $(B_4=B\#4\overline{\CP}^2,\w')$.
     The gluing construction along $e_{0123}$ with $\CP^2$ along a quadric shows that the resulting
     symplectic open manifold is $B_3$ with the correct symplectic
     areas.  One way of verification is that, one can regard all
     these surgeries being supported in the complement of a line in
     $\CP^2$, which gives $\CP^2\#3\overline{\CP}^2$ by \cite[Theorem
     1.1]{Do10a} and \cite[Theorem 1.1]{HK99}.  But removing a line from
     $\CP^2\#3\overline{\CP}^2$ clearly gives $B_3$.  The symplectic
     areas are checked directly from the construction in Lemma
     \ref{l:basis}. For example, consider $\w(u_1)$.  It consists of the
     gluing of $e_1$, $e_2$ and $e_{12}$.  Notice $e_{12}$ is a line in the $\CP^2$
     used for rational blow-down. Therefore, it has half the area of the quadric,
     which has equal area of $e_{0123}$.  It follows that
     $$\w(u_1)=\w'(e_1)+\w'(e_2)+\w'(e_{12})=(\tfrac{\alpha}{2}-\epsilon)+(\tfrac{\alpha}{2}-\epsilon)+
     \tfrac{1}{2}4\epsilon=\alpha$$
     so the exceptional sphere has the correct area.
\epf

\brmk\label{r:LCAS}
     Notice that from the above construction, the resulting
     Lagrangian $\RP^2$ has the three exceptional spheres as its
     associated basis from Lemma~\ref{l:basis}.
\ermk

\subsubsection{Proof of  Proposition \ref{p:KRP2}}
\bpf[Proof of  Proposition \ref{p:KRP2}]

We only need to prove this for $k=9$.  For larger $k$ the argument carries over with no further efforts.
For part (1) by Lemma \ref{l:CRP2} and Remark~\ref{r:LCAS}
we can construct a Lagrangian $L$ which is an embedded $\RP^{2}$ in
certain $(M=\CP^2\#9\overline{\CP}^{2},\w')$ such that its associated basis from
Lemma~\ref{l:basis} is $\{H, U_{1}, U_{2}, U_{3}, E_{4}, \dots, E_{9}\}$
and its $\Z_{2}$-homology class is $[L] = U_{1} + U_{2} + U_{3}$. 

Since this can be done such that $U_{1}, U_{2}, U_{3}, E_{4}, \dots, E_{9}$
have $\w'$-area as small as we like by the construction in Lemma \ref{l:CRP2}, we may assume that
$$E = 3H - 2U_{1} - E_{4} - \cdots - E_{9}$$
is represented by an exceptional symplectic sphere in $(M, \w')$.  Note that $E \cdot [L] = 0$ in
$\Z_{2}$-homology, so we can apply part (2) in Lemma~\ref{l:HERB} to see that
$$
\iota(E) = 3H' + E_{0}' - E_{1}' - E_{2}' - E_{3}' - E_{4}' - \cdots - E_{9}' \in H_{2}(W; \Z).
$$
It is well-known this class has minimal symplectic genus equal $1$, which goes back to Kervaire--Milnor
\cite{KM61},
see also \cite[(3.1) on pp.\ 3]{Li03}.
Therefore since $\iota(E)$ is not an exceptional class, it follows from part (3) in Lemma~\ref{l:HERB}
that $E$ does not have a symplectic representative disjoint from $L$.

     For part (2) consider $(M'=M\#\overline{\CP}^2,\w)$, which is a symplectic blow-up of
     $M$ away from $L$.  The new exceptional class and sphere will be denoted as $E_{10}$
     and $e_{10}$, respectively. Note that $e_{10}\cap L=\emptyset$.  Now extend $\{E, E_{10}\}$ to a homology basis
     consisting of a line class and exceptional classes otherwise.
     This can always be done, for example, by taking the Cremona transforms which sends $E_1$ to $E$ and fixes
     $E_{10}$ (such Cremona transforms can always be found, see \cite{MSc12}) and act
     on the original basis $\{H,E_1\dots E_{10}\}$.

     Let $F$ be a diffeomorphism $F:M'\rightarrow M'$ inducing the Cremona transform
     that is reflection along $E - E_{10}$, then $F_*$ switches $E$ and $E_{10}$ and acts trivially
     on the rest of basis elements.
     Clearly $L'=F(L)\hookrightarrow (M',(F^{-1})^*\w)$ is a Lagrangian embedding disjoint from
     the spherical representative $F(e_{10})$ of class $E$.  By the form
     of $F_*$ and the fact that $[L]\cdot E=[L]\cdot E_{10}=0$ in $\Z_2$-pairing,
     we see that $L'$ is $\Z_2$-homologous to $L$.

     To see that $L$ and $L'$ are not smoothly isotopic, we assume
     the contrary.  Such an isotopy can be extended to a family of
     diffeomorphism $f_t$ for $t\in[0,1]$ of $M'$.  Let $e$ be a
     representative of $E$ with $e\cap L'=\emptyset$, for example, $e=F(e_{10})$.  Then we have that $f_1^{-1}(e)\cap
     L=\emptyset$.  This is a contradiction to part (1).
\epf



\section{Symplectic del Pezzo surfaces}\label{s:del Pezzo}

\subsection{Lagrangian $\RP^{2}$'s in del Pezzo surfaces}

Let $(X_{k}, \w_{0})$ be a symplectic del Pezzo surface, i.e.\ a monotone $\CP^{2}\#k \overline{\CP}^{2}$ built by blowing up $k \leq 8$ disjoint balls $B(\tfrac{1}{3}) \subset (\CP^{2}, \w)$ where
$\w$ is normalized so that $\int_{\CP^{1}}\w = 1$.

\begin{prop} \label{p:pafRP}   There is a symplectic packing of $8$ balls $B(\tfrac{1}{3})$ into
    $(\CP^{2} \setminus \RP^{2}, \w)$ and therefore there is a Lagrangian $\RP^{2}$ in any
    del Pezzo surfaces $(X_{k}, \w_{0})$ for $k \leq 8$.
\end{prop}
\bpf
    By Lemma~\ref{l:EQBP} below it suffices to pack $8$ balls of size $\frac{1}{3}$
    into $(S^2\times S^2,\Omega_{1,\frac{1}{2}})$.
    Since blowing up a ball of size $\tfrac{1}{3}$  in
    $(S^2\times S^2,\Omega_{1,\frac{1}{2}})$ leads to
    $(\CP^{2}\#2 \overline{\CP}^{2}, \w')$ with $\w'$ dual to the class
    $\frac{7}{6}H-\frac{2}{3}E_1-\frac{1}{6}E_2$, it suffices to prove that
    the vector
    $$
        (7|4,1,2,2,2,2,2,2,2) = 7H - 4E_{1} - E_{2} - 2 \sum_{i=3}^{9} E_{i}
    $$
    is Poincare dual to a symplectic form for $\CP^{2}\#9 \overline{\CP}^{2}$.

    Permuting the $E_{j}$ and applying the Cremona transform along
    $H-E_1-E_2-E_3$ three times, transforms $(7|4,1,2,2,2,2,2,2,2)$
    into $(4|2, 1, 1, 1, 1, 1, 1, 1, 1)$, which is a reduced vector and it follows from the criterion in
    \cite[Lemma 4.7(2)]{LL01}
    that it represents a symplectic form.  Similar arguments appeared in \cite{BP12}.
\epf

Let $\Omega_{\alpha,\beta} = \alpha \sigma_{A} \oplus
\beta \sigma_{B}$ be the symplectic form on $S^{2}\times S^{2}$
where $\sigma_{A}, \sigma_{B}$ are area forms on
each factor with area $1$.
When $\alpha > 2\beta$, there is a symplectomorphism between
$(S^{2} \times S^{2}, \Omega_{\alpha, \beta})$ and the Hirzebruch surface $(H_{4}, \Omega_{\alpha, \beta})$.
Here $H_{4} = \bP(\mathcal{O}(4) \oplus \C) \to \CP^{1}$, the fiber has area $\beta$,
the zero section $Z_{0} = \bP(0 \oplus \C)$ has area $\alpha + 2\beta$,
and the section at infinity $Z_{\infty} = \bP(\mathcal{O}(4) \oplus 0)$ has area $\alpha - 2\beta$,
and the symplectomorphism identifies $Z_{\infty}$ with a symplectic surface $Z \subset S^{2} \times S^{2}$
in class $A-2B$.

\blem\label{l:EQBP}
    Symplectic ball packing for
    $(\CP^2\backslash\RP^2,\w)$ and
    $(S^2\times S^2,\Omega_{1,\frac{1}{2}})$ are equivalent.
\elem

\bpf
    It follows for the Biran decomposition \cite[Theorem 1.A]{Bi01}
    associated to the quadric $Q \subset \CP^{2}$ \cite[Section 3.1.2]{Bi01}, that there is an open
    set in $(\CP^{2} \setminus \RP^{2}, \w)$ symplectomorphic to
    $(H_{4} \setminus Z_{\infty}, \Omega_{\alpha, \beta})$
    where $\alpha>2\beta$ and $\alpha+2\beta=2$.
    Hence there is a symplectic embedding of
    $(S^2\times S^2 \backslash Z, \Omega_{\eps})$ into
    $(\CP^2\backslash \RP^2,\w)$, where $\Omega_{\eps} = \Omega_{\alpha, \beta}$
    with $\alpha=1+2\epsilon$ and
    $\beta=\frac{1}{2}-\epsilon$.

    If $\phi:\coprod B_i\hookrightarrow (\CP^2\backslash \RP^2, \w)$
    is a symplectic packing, then by the openness of
    symplectic ball-packing we also have an packing of $(1+\delta) \coprod B_i$
    for some $\delta > 0$.  Performing a symplectic cut around $\RP^{2}$ gives
    a symplectic packing of $(1+\delta) \coprod B_i$ into
    $(S^2\times S^2,\Omega_{\eps})$ and by rescaling we have
    a symplectic packing of $\coprod B_i$ into $(S^2\times S^2,\tfrac{1}{1+\delta}\Omega_{\eps})$.
    Taking $\eps = \delta/2$ gives a packing of $(S^{2} \times S^{2}, \Omega_{1,c})$ where $c < 1/2$
    and hence a packing of $\Omega_{1,\frac{1}{2}}$.

 Conversely, let $\phi:\coprod B_i\hookrightarrow (S^2\times
    S^2,\Omega_{1,\frac{1}{2}})$ be a symplectic packing and form a blow-up using $\phi$.
    Standard Gromov--Witten theory shows that one may find
    a set of exceptional spheres in the complement of a curve in class $B$ of the blown-up manifold.
    This implies the packing can be performed in the complement of a curve in class $B$ and hence a packing of
    $(S^2\times S^2,\Omega_{1,\frac{1}{2}-\eps})$ for some $\eps>0$.  This in turn gives
    a packing of $(S^2\times S^2,\Omega_{\eps})$.
    Using the relative inflation methods of \cite{Bi99} gives
    a packing of $(S^2\times S^2,\Omega_{\eps})$ in the complement of the $(-4)$-sphere
    $Z$ and hence into an open subset of $(\CP^2\backslash \RP^2, \w)$.
\epf

\subsection{The Proof of Theorem~\ref{t:RPNS} and applications}\label{s:final remarks}

The key to proving Theorem~\ref{t:RPNS} is the following lemma.

\begin{lem}\label{l:sdfRP}
    Let $(\overline{M}^{4}, \overline{\w})$ be a closed rational or ruled symplectic manifold
    with rational symplectic form $[\overline{\w}] \in H^{2}(\overline{M}; \Q)$.
    Suppose that
    $L \subset (\overline{M},\overline{\w})$ is a Lagrangian $\RP^{2}$ or $S^2$, there are two orthogonal
    exceptional classes $E_{1}, E_{2} \in \mathcal{E}_{\overline{\w}}$ with the same symplectic
    area $a$, and $\langle[E_i],[L]\rangle_{\Z}=0$.  Then there is a Lagrangian sphere
    $L_{0} \subset (\overline{M}, \overline{\w})$ in class
    $E_{1} - E_{2}$ that is disjoint from $L$.
\end{lem}
\begin{proof}
    Notice that the homological pairing condition is void when $L=\RP^2$.  We only give the proof
    for $\RP^2$, and the case for $L=S^2$ is identical with the trivial pairing assumption.

    By Proposition~\ref{p:DfRP} we can pick representatives $C_{i}$ for the classes $E_{i}$
    that are disjoint from $L$.  Hence performing a symplectic cut on
    a neighborhood of $L$ produces $(\overline{W}, \bar{\w})$ with cut divisor $Z$
    disjoint from the exceptional spheres $C_{i}$.
    Let $(W, \w)$ be the result of blowing down the two curves
    $C_{i}$ to balls $B_{i}(a)$, so that
    $B_i(a)$ are disjoint from $Z \subset W$, and let $p: \overline{W} \to W$ be
    the topological blow down map.

    Taking sub-balls $B_{i}(\eps) \subset B_{i}(a)$ for $\epsilon$ sufficiently small, we can
    assume there is a Darboux chart that is disjoint from $Z$ and contains both $B_{i}(\eps)$.
    If $(\overline{W}, \bar{\w}_{\eps})$ is the result of blowing-up these two balls, it
    it follows from the local construction in the proof of \cite[Lemma 5.4]{LW12} that there is
    a Lagrangian sphere $L_{0}' \subset (\overline{W}, \bar{\w}_{\eps})$ that is disjoint from $Z$ and in class
    $E_{1} - E_{2}$.

    Let $B_{b} = PD(p^{*}[\w]) - b\,(E_{1} + E_{2})$ where $b$ is a rational number slightly
    larger than $a$, then by
    \cite[Lemma 5.1]{LW12} and \cite[Lemma 2.2.B]{Bi99}
    for
    $q \gg 0$ sufficiently large we have
    a symplectic surface $C \subset (\overline{W}, \bar{\w}_{\eps})$
    such that $[C] = qB_{b}$, is disjoint from
    the Lagrangian sphere $L_{0}'$, and intersects $Z$ transversally.
    Biran's relative inflation method \cite[Proposition 2.1.A]{Bi99} along $C$ now
    builds a symplectic form $\bar{\w}' \in [\bar{\w}]$ on $\overline{W}$ such that
    $L_{0}'$ is $\bar{\w}'$-Lagrangian and $Z$ is $\bar{\w}'$-symplectic.
Since by construction $\bar{\w}'$ and $\bar{\w}$ are deformation equivalent, through symplectic forms for which $Z$ is a symplectic manifold it follows from \cite[Proposition 1.2.9]{Mc12} that
there is an symplectomorphism
$\phi:(\overline{W},\bar{\w}') \to (\overline{W}, \bar{\w})$ that preserves $Z$.  Hence $L_{0} = \phi(L_{0}')$
is a Lagrangian sphere in $(\overline{W}, \bar{\w})$ that is disjoint from $Z$ and is in homology class
$E_{1} - E_{2}$.  Gluing a neighborhood of $\RP^{2}$ back in for $Z$ gives the desired result.
\end{proof}

\begin{proof}[Proof of Theorem~\ref{t:RPNS}]
    Since we are in the monotone case, all Lagrangian spheres $S^{2} \subset X_{k}$
    are heavy with respect to the fundamental class $\ide \in QH_{4}(X_{k}; \Z/2\Z)$
    by \cite[Theorem 1.17]{EP09}.  So to show that a given Lagrangian $L$ is not superheavy
    with respect to the fundamental class $\ide$, by \cite[Theorem 1.4]{EP09} it suffices
    to build a Lagrangian sphere $L'$ disjoint from $L$.

    Let $L$ be a Lagrangian $\RP^{2}$ in $(X_{k}, \w_{0})$ with $k \geq 2$, then
    by Lemma~\ref{l:sdfRP} there is a non-trivial Lagrangian
    sphere $L_{0} \subset X_{k}$ that is disjoint from $L$.

    Let $L$ be a Lagrangian $S^{2}$ in $(X_{k}, \w_{0})$ with $k \geq 3$.
    By \cite[Theorem 1.4]{LW12} and Lemma~\ref{l:sdfRP} there are
    Lagrangian spheres $L_{1}$ and $L_{2}$ that are disjoint from each other.
    Using \cite[Theorem 1.8]{LW12} and
    Corollary \ref{c:TTS2}, we have a symplectomorphism $\psi$ such that $\psi(L_2)=L$
    and therefore $\psi(L_1)$ is a Lagrangian $S^{2}$ disjoint from $L$.
\end{proof}

\begin{remark}
     The case when $k=2$ for Lagrangian spheres can be proved by adapting the computation of Fukaya--Oh--Ohta--Ono \cite{FOOO} in the $S^2\times S^2$ case to $S^2\times S^2\#\overline{\CP}^{2}=(X_2,\w_2)$.  Therefore
     combining the result for $k=0$ proved in \cite{Wu12}, monotone spherical Lagrangians in rational surfaces are proved to not be superheavy, with the exception of $\RP^2\subset \CP^2\#\overline{\CP}^{2}$, which presumably can be done using techniques in \cite{Wu12}.
\end{remark}



     We discuss an application of Theorem \ref{t:RPNS} of potential
     interest.
     We will consider the toric degeneration picture Figure \ref{fig:side:a} of
     $S^2\times S^2$ that first appeared in \cite{FOOO}, but let us note that
     it is possible for one to consider other semi-toric systems as well (cf.\ \cite{Wu12}).

\begin{figure}[h]
\begin{minipage}[t]{0.5\linewidth}
\centering
\includegraphics[width=2.4in]{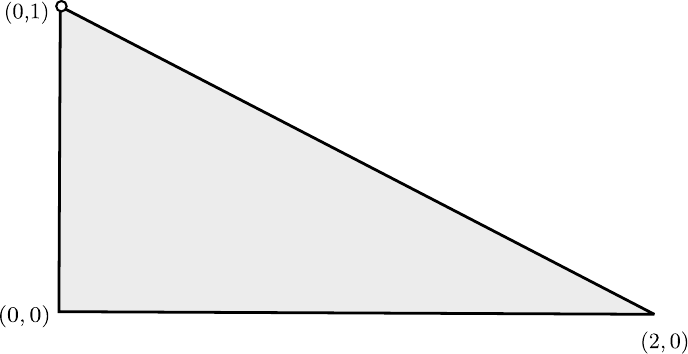}
\caption{$S^2\times S^2$} \label{fig:side:a}
\end{minipage}%
\begin{minipage}[t]{0.5\linewidth}
\centering
\includegraphics[width=2.4in]{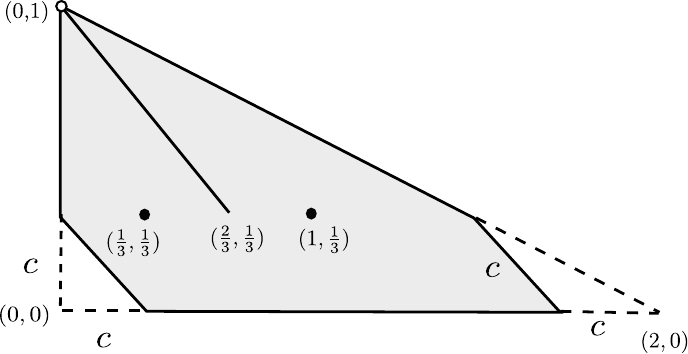}
\caption{$S^2\times S^2\#2\overline{\CP}^{2}$} \label{fig:side:b}
\end{minipage}
\end{figure}

The reader should consider Figure 1 as an ``open polytope" in the sense
that there is no actual toric action on $S^2\times S^2$ with the
polytope in Figure \ref{fig:side:a}.  Instead, there is such a toric
action on the complement of the anti-diagonal $S^2\times
S^2\backslash \bar{\Delta}$, whose moment map has image equal the polytope
in Figure 1 with the white dot removed.

In this open toric picture, which is closely related to the \textit{semi-toric} case
considered by \cite{PN09}, one could still perform two toric blow-ups of size $c$
by standard ``corner chopping", giving Figure 2 as moment polytope.  When $c=\frac{1}{3}$, the resulting
symplectic manifold $S^2\times S^2\#2\overline{\CP}^2$
is monotone and symplectomorphic to $X_3$ in
Theorem \ref{t:RPNS}.
Since the proper transform of the
Lagrangian $\bar{\Delta}\subset S^2\times S^2$ is not
superheavy
it follows that some fiber in the open polytope
must be non-displaceable by \cite[Theorem 1.8]{EP09}.
One may further narrow down the position of these fibers using
probes invented by McDuff \cite{Mc11}.  It is not hard to check
that there are three sets of potentially non-displaceable fibers as
shown by the solid line and two extra dots in Figure
\ref{fig:side:b}.  What we showed is that at least one of these
points is actually non-displaceable.


\appendix

\section{Lagrangian spheres in rational manifolds}

\subsection{Characteristic Lagrangian spheres}

We first recall from \cite[Definition 3.3]{LW12} that a \textit{stable spherical symplectic configuration} is an ordered configuration of symplectic spheres with: (1) $c_1\geq1$ for all irreducible
components, (2) the intersection numbers between two different components are 0 or 1, (3) they are
simultaneously holomorphic with respect to some almost complex structure $J$ tamed by the symplectic form.
We will call them stable configurations for brevity.  In the proof of \cite[Theorem 1.5]{LW12}, 
the following intermediate result is reached.

\blem\label{l:LW1.5} In $\CP^2\#4\overline{\CP}^2$, let $L_1$ and $L_2$ be Lagrangian spheres
in the homology class $E_{1} - E_{2}$ and suppose they are disjoint from a
stable configuration with irreducible components in classes $\{H-E_1-E_2, H-E_3-E_4, E_3, E_4\}$, then
$L_{1}$ and $L_{2}$ are Hamiltonian isotopic in the complement of the stable configuration.
\elem

In particular in the proof of \cite[Theorem 1.5]{LW12} one uses \cite[Proposition 6.8]{LW12}
to show that $L_{1}$ and $L_{2}$ are Hamiltonian isotopic in the complement of the stable configuration.
The same holds true for $\CP^2\#(k+1)\overline{\CP}^2$ as well for $k=1,2$ with the stable
configurations specified in \cite{LW12}.

\begin{thm}\label{t:char}
Lagrangian $S^2$'s in a symplectic rational manifold
with $\chi \leq 7$
are unique up to Hamiltonian isotopy.
\end{thm}
\bpf
    By \cite[Theorem 1.5 and Proposition 4.10]{LW12},
     we only need to deal with the case where $M=\CP^2\#3\overline{\CP}^2$ and
     $[L_i]=H-E_1-E_2-E_3$ for $i=1,2$.

     Fix a Darboux chart $U_p\subset M$ that is disjoint from $L_{1} \cup L_{2}$
     and centered at the point $p \in M$.
     By blowing-up a symplectically embedded ball $B_p\subset U_p$,
     we can build a symplectic manifold $(M'=\CP^2\#4\overline{\CP}^2,\w')$
     with a exceptional sphere $C$ such that
     $H_2(M';\Z)$ has a basis identified
     with the union of a basis of $H_2(M,\Z)$ and $[C]$, the intersection product
      $[L_{i}] \cap [C] = 0$.

      From the Gromov--Taubes invariant theory, for generic compatible almost complex structure $J$
      the classes $H-E_1-[C]$, $H-E_2-[C]$, and $H-E_3-[C]$
      have unique representatives as $J$-holomorphic exceptional spheres
      $C_1$, $C_2$ and $C_3$, respectively, which are disjoint.
      Since $[C_{i}] \cap [L_{j}] = 0$, \cite[Corollary 3.13]{LW12} builds Hamiltonian
      isotopies $\psi_{j}$ so that $\psi_{j}(L_{j})$ is disjoint from $C_{1} \cup C_{2} \cup C_{3}\cup C$.

     Notice that the set of classes $\{H-E_1-[C], H-E_2-[C], H-E_3-[C], [C]\}$ are Cremona equivalent to
     $\{H-E_1-E_2, H-E_3-E_4, E_3, E_4\}$, Lemma \ref{l:LW1.5} applies.
     It follows that $L_1$ and $L_2$ are Lagrangian isotopic in the complement of a neighborhood of $C\cup\bigcup C_i$ in $(M',\w')$, in particular the complement of $C$ which is symplectomorphic to an open set of $M$.
\epf

\subsection{Proof of Corollary \ref{c:TTS2}} \label{s:A2}

\begin{proof}
Part (2) follows from Theorem \ref{t:char} and \cite[Theorem 1.6]{LW12}.
When $\chi(M) = 6$ and the homology class of the Lagrangians is characteristic, then
Theorem~\ref{t:char} covers part (1).  In all the rest of cases, we
assume that
$[L_i]=E_1-E_2$ without loss of generality by \cite[Proposition 4.10]{LW12}.
Our proof follows the steps sketched in \cite{LW12}.

For each pair $(M, L_i)$ by \cite[Theorem 1.1]{LW12}, away from $L_i$ there is a set of disjoint
$(-1)$~symplectic spheres $C^l_i$ for $l=3,\dots, k+1,$
with
$$
\mbox{$[C_i^l]=E_l$ for $l=3, \dots, k$ and  $[C^{k+1}_i]= H-E_1-E_2$.}
$$
Blowing down the collections $\mathcal{C}_{i} = (C^{3}_{i}, \dots, C^{k+1}_{i})$
separately, results in
$(\tilde{M_i}, \tilde L_i, \mathcal{B}_{i})$ where $\tilde{M_i}$ is a symplectic $S^2\times S^2$ with equal symplectic
areas in each factor, $\tilde L_i$ a
Lagrangian sphere, and $\mathcal{B}_i = (B_{i}^3, \dots, B_{i}^{k+1})$ is a symplectic ball packing
in $\tilde{M}_{i} \backslash \tilde{L}_{i}$
corresponding to $\mathcal{C}_{i}$.

By Lalonde--McDuff \cite{LM2} and Hind \cite{Hi04},
there is a symplectomorphism between the
pairs
$\Psi: (\tilde{M_1}, \tilde L_1) \to (\tilde{M_2}, \tilde L_2)$.
For fixed $l$, the symplectic balls $\Psi(B_{1}^{l})$ and $B_{2}^{l}$ have the same volume
since they come from blowing down the same class.  Hence by
Theorem~\ref{t:BPRL-sphere} there is a compactly supported
Hamiltonian isotopy $\Phi$ of $\tilde{M}_{2}\setminus L_{2}$
connecting the symplectic ball packing $\Psi(\mathcal{B}_1) = \{\Psi(B_{1}^{l})\}_{l}$ and
$\mathcal{B}_{2}$ in $\tilde{M}_{2}\setminus L_{2}$.
Therefore $\Phi \circ \Psi$ is a symplectomorphism between the tuples
$(\tilde{M_i}, \tilde L_i, \mathcal{B}_i)$ and hence upon blowing up induces a
symplectomorphism
$$\psi: (M, L_{1}, \mathcal{C}_{1}) \to (M, L_{2}, \mathcal{C}_{2}).$$
By design $\psi$ preserves the homology classes $E_{1} - E_{2}$, $H-E_{1}-E_{2}$,
$E_{3}, \dots, E_{k}$ as well as the class $[\w]$, from which it follows that $\psi$ it also preserves $H$
and hence $\psi \in \Symp_{h}(M)$.
\end{proof}


\section{$\Z_2$-orthogonal systems and push off systems}

Recall that in the proof of Lemma \ref{l:counting} for $1 \leq k \leq 8$ we need to compare the counts for:

\begin{itemize}
  \item $\Z_{2}$-orthogonal systems for $H$, i.e.\ sets of $k$ exceptional classes that
  are $\Z_2$-orthogonal to $H$ 
  in $\CP^{2}\# k\overline{\CP}^{2}$.
  \item Push off system for $S$, i.e.\ sets of $k$ exceptional classes that are $\Z$-orthogonal to $S=-H+2E_1-E_2$ 
  in $\CP^{2}\# (k+1)\overline{\CP}^{2}$.
\end{itemize}
For ease of notation
when $k\geq4$ we will replace $S$ with the Cremona equivalent
$S' = E_1-E_2-E_3-E_4$.
In the following $\mathcal{O}_{k} = \{E_i\}_{i=1}^k$ will be a $\Z_2$-orthogonal system
for $H$ in $\CP^{2}\#k\overline{\CP}^{2}$, while
$\mathcal{P}_{k} = \{H-E_1-E_i\}_{i=2}^{4}\cup \{E_j\}_{j=5}^{k+1}$ will be
a push off system for $S'$ in $\CP^{2}\#(k+1)\overline{\CP}^{2}$.

\begin{description}
\item[For $k \leq 3$:] 
$H$ has $\mathcal{O}_{k}$ and $S$ has
$\{H-E_1-E_2\} \cup \{E_{i}\}_{i=3}^{k+1}$.

\item[For $4 \leq k \leq 5$:] 
$H$ has $\mathcal{O}_{k}$ and $S'$ has
 $\mathcal{P}_{k}$.

\item[For $k=6$:] 
$H$ has $\mathcal{O}_{6}$ and its Cremona transform along
$2H-\sum_{i=1}^6 E_i$.  While $S'$ has $\mathcal{P}_{6}$ and its Cremona transform along $H-E_5-E_6-E_7$.

\item[For $k=7$:] There are $8$ systems: $H$ has $\mathcal{O}_{7}$ and
its $7$ Cremona transforms along
$$2H-(E_1+\cdots+E_7)+E_i \quad\mbox{for}\quad 1\leq i\leq 7.$$
While $S'$ has $\mathcal{P}_{7}$, its $4$ Cremona transforms along
$$H-(E_5+E_6+E_7+E_8)+E_j \quad\mbox{for}\quad j=5, 6, 7, 8$$
and its $3$ Cremona transforms along
$$2H-E_1-(E_5+E_6+E_7+E_8)+E_j \quad\mbox{for}\quad j=2, 3,4.$$

\item[For $k = 8$:] There are $29$ systems:
$H$ has $\mathcal{O}_{8}$ and its $28$ Cremona transforms along
$$2H-(E_1+\cdots+E_7+E_8)+E_i+E_j \quad \mbox{for}\quad 1\leq i\ne j\leq 8. $$
While $S'$ has $\mathcal{P}_{8}$, its $10$ Cremona transforms along
$$H-(E_5+E_6+E_7+E_8+E_9)+E_i+E_j \quad \mbox{for}\quad 5\leq i\ne j\leq 9,$$
its $15$ Cremona transforms along
$$2H-E_1-(E_5+E_6+E_7+E_8+E_9)+E_j+E_p \quad\mbox{for}\quad 2\leq j\leq 4,\, 5\leq p\leq 9,$$
and its $3$ Cremona transforms along
$$3H-2E_1-(E_2+E_3+E_4)-(E_5+E_6+E_7+E_8+E_9)+E_q \quad\mbox{for} \quad 2\leq q\leq 4.$$
 \end{description}

\bigskip

\noindent
\begin{tabular}{lll}
Matthew Strom Borman & Tian-Jun Li & Weiwei Wu \\
University of Chicago &  University of Minnesota & Michigan State University \\
borman@math.uchicago.edu & tjli@math.umn.edu & wwwu@math.msu.edu
\end{tabular}


\begin{thebibliography}{99}

\bibitem{Au07}
M.~Audin.
\newblock Lagrangian skeletons, periodic geodesic flows and symplectic
  cuttings.
\newblock {\em Manuscripta Math.}, 124(4):533--550, 2007.

\bibitem{Bi99}
P.~Biran.
\newblock A stability property of symplectic packing.
\newblock {\em Invent. Math.}, 136(1):123--155, 1999.

\bibitem{Bi01}
P.~Biran.
\newblock Lagrangian barriers and symplectic embeddings.
\newblock {\em Geom. Funct. Anal.}, 11(3):407--464, 2001.

\bibitem{BP12}
O.~Buse and M.~Pinsonnault.
\newblock Packing numbers of rational ruled 4-manifolds.
\newblock {\em J. Symplectic Geom. (to appear)}, 2013.
\newblock arXiv:1104.3362v1.

\bibitem{DRE12}
G.~Dimitroglou Rizell and J.D.~Evans.
\newblock Unlinking and unknottedness of monotone {L}agrangian submanifolds.
\newblock arXiv:1211.6633v1, 2012.

\bibitem{Do10a}
J.G.~Dorfmeister.
\newblock Kodaira dimension of fiber sums along spheres.
\newblock arXiv:1008.4447v1, 2010.

\bibitem{Do10b}
J.G.~Dorfmeister.
\newblock Minimality of symplectic fiber sums along spheres.
\newblock {\em Asian J. Math. (to appear)}, 2013.
\newblock arXiv:1005.0981v1.

\bibitem{FS04}
R.~Fintushel and R.J. Stern.
\newblock Invariants for {L}agrangian tori.
\newblock {\em Geom. Topol.}, 8:947--968 (electronic), 2004.

\bibitem{EP97}
Y.~Eliashberg and L.~Polterovich.
\newblock The problem of {L}agrangian knots in four-manifolds.
\newblock In {\em Geometric topology ({A}thens, {GA}, 1993)}, volume~2 of {\em
  AMS/IP Stud. Adv. Math.}, pages 313--327. Amer. Math. Soc., Providence, RI,
  1997.

\bibitem{EP09}
M.~Entov and L.~Polterovich.
\newblock Rigid subsets of symplectic manifolds.
\newblock {\em Compos. Math.}, 145(3):773--826, 2009.


\bibitem{Ev10}
J.D.~Evans.
\newblock Lagrangian spheres in del {P}ezzo surfaces.
\newblock {\em J. Topol.}, 3(1):181--227, 2010.

\bibitem{Ev11}
J.D.~Evans.
\newblock Symplectic mapping class groups of some {S}tein and rational
  surfaces.
\newblock {\em J. Symplectic Geom.}, 9(1):45--82, 2011.

\bibitem{FOOO}
K.~Fukaya, Y.-G.~Oh, H.~Ohta and K.~Ono.
\newblock Toric Degeneration and nondisplaceable Lagrangian tori in $S^2 \times S^2$.
\newblock {\em Int. Math. Res. Not. IMRN}, 2012(13):2942--2993, 2012.

\bibitem{Go95}
R.E.~Gompf.
\newblock A new construction of symplectic manifolds.
\newblock {\em Ann. of Math. (2)}, 142(3):527--595, 1995.

\bibitem{HK99}
J.-C.~Hausmann and A.~Knutson.
\newblock Cohomology rings of symplectic cuts.
\newblock {\em Differential Geom. Appl.}, 11(2):197--203, 1999.

\bibitem{Hi04}
R.~Hind.
\newblock Lagrangian spheres in {$S^2\times S^2$}.
\newblock {\em Geom. Funct. Anal.}, 14(2):303--318, 2004.

\bibitem{Hi12}
R.~Hind.
\newblock Lagrangian spheres in Stein surfaces.
\newblock {\em Asian J. Math.}, 16(1):1--36, 2012.

\bibitem{IP04}
E.-N.~Ionel and T.H.~Parker.
\newblock The symplectic sum formula for Gromov-Witten invariants.
\newblock {\em Ann. of Math. (2)} 159(3):935--1025, 2004.

\bibitem{KM61}
M.~Kervaire and J.W.~Milnor.
\newblock On $2$-spheres in $4$-manifolds.
\newblock {\em Proc. Nat. Acad. Sci. USA}, 47:1651--1657, 1961.

\bibitem{KS02}
M.~Khovanov and P.~Seidel.
\newblock Quivers, Floer cohomology, and braid group actions. 
\newblock {\em J. Amer. Math. Soc.}, 15(1):203--271, 2002.

\bibitem{La94}
F.~Lalonde.
\newblock Isotopy of symplectic balls, {G}romov's radius and the structure of
  ruled symplectic {$4$}-manifolds.
\newblock {\em Math. Ann.}, 300(2):273--296, 1994.

\bibitem{LM2}
F.~Lalonde and D.~McDuff.
\newblock {$J$}-curves and the classification of rational and ruled symplectic
  {$4$}-manifolds.
\newblock In {\em Contact and symplectic geometry ({C}ambridge, 1994)},
  volume~8 of {\em Publ. Newton Inst.}, pages 3--42. Cambridge Univ. Press,
  Cambridge, 1996.

\bibitem{Le95}
E.~Lerman.
\newblock Symplectic cuts.
\newblock {\em Math. Res. Lett.}, 2(3):247--258, 1995.

\bibitem{Li02}
B.-H.~Li and T.-J.~Li.
\newblock Symplectic genus, minimal genus and diffeomorphisms.
\newblock {\em Asian J. Math.}, 6(1):123--144, 2002.

\bibitem{Li03}
B.-H.~Li and T.-J.~Li.
\newblock Smooth minimal genera for small negative classes in $\CP^2\#n\overline{\CP^2}$ with $n\leq9$.
\newblock {\em Topology Appl.}, 6(1):123--144, 2002.

\bibitem{LL01}
T.-J.~Li and A.-K.~Liu.
\newblock Uniqueness of symplectic canonical class, surface cone and symplectic
  cone of 4-manifolds with {$B^+=1$}.
\newblock {\em J. Differential Geom.}, 58(2):331--370, 2001.

\bibitem{LR01}
A.-M.~Li and Y.~Ruan.
\newblock Symplectic surgery and Gromov-Witten invariants of Calabi-Yau 3-folds.
\newblock {\em Invent. Math.}, 145(1):151--218, 2001.

\bibitem{LW12}
T.-J.~Li and W.~Wu.
\newblock Lagrangian spheres, symplectic surfaces and the symplectic mapping
  class group.
\newblock {\em Geom. Topol.}, 16(2):1121--1169, 2012.

\bibitem{Liu96}
A. ~Liu.
\newblock Some new applications of the general wall crossing formula.
\newblock {\em Math. Res. Lett.}, 3:569--585, 1996.

\bibitem{Mc90}
D.~McDuff.
\newblock The structure of rational and ruled symplectic {$4$}-manifolds.
\newblock {\em J. Amer. Math. Soc.}, 3(3):679--712, 1990.

\bibitem{Mc91}
D.~McDuff.
\newblock Blow ups and symplectic embeddings in dimension {$4$}.
\newblock {\em Topology}, 30(3):409--421, 1991.

\bibitem{Mc93}
D.~McDuff.
\newblock Remarks on the uniqueness of symplectic blowing up.
\newblock In {\em Symplectic geometry}, volume 192 of {\em London Math. Soc.
  Lecture Note Ser.}, pages 157--167. Cambridge Univ. Press, Cambridge, 1993.

\bibitem{Mc97}
D.~McDuff.
\newblock Lectures on {G}romov invariants for symplectic {$4$}-manifolds.
\newblock In {\em Gauge theory and symplectic geometry ({M}ontreal, {PQ},
  1995)}, volume 488 of {\em NATO Adv. Sci. Inst. Ser. C Math. Phys. Sci.},
  pages 175--210. Kluwer Acad. Publ., Dordrecht, 1997.

\bibitem{Mc98}
D.~McDuff.
\newblock From symplectic deformation to isotopy.
\newblock In {\em Topics in symplectic {$4$}-manifolds ({I}rvine, {CA}, 1996)},
  First Int. Press Lect. Ser., I, pages 85--99. Int. Press, Cambridge, MA,
  1998.

\bibitem{Mc11}
D.~McDuff.
\newblock Displacing {L}agrangian toric fibers via probes.
\newblock In {\em Low-dimensional and symplectic topology}, volume~82 of {\em
  Proc. Sympos. Pure Math.}, pages 131--160. Amer. Math. Soc., Providence, RI,
  2011.

\bibitem{Mc12}
D.~McDuff.
\newblock Nongeneric {$J$}-holomorphic curves in rational manifolds.
\newblock arXiv:1211.2431v1, 2012.

\bibitem{MP94}
D.~McDuff and L.~Polterovich.
\newblock Symplectic packings and algebraic geometry.
\newblock {\em Invent. Math.}, 115(3):405--434, 1994.
\newblock With appendix by Y. Karshon.


\bibitem{MSc12}
D.~McDuff and F.~Schlenk.
\newblock The embedding capacity of 4-dimensional symplectic ellipsoids.
\newblock {\em Ann. of Math. (2)}, 175(3):1191--1282, 2012.


\bibitem{PN09}
A.~Pelayo and S.~V{\~u}~Ngoc.
\newblock Semitoric integrable systems on symplectic 4-manifolds.
\newblock {\em Invent. Math.}, 177(3):571--597, 2009.

\bibitem{Se08}
P.~Seidel.
\newblock Lectures on four-dimensional {D}ehn twists.
\newblock In {\em Symplectic 4-manifolds and algebraic surfaces}, volume 1938
  of {\em Lecture Notes in Math.}, pages 231--267. Springer, Berlin, 2008.

\bibitem{Vi06}
S.~Vidussi.
\newblock Lagrangian surfaces in a fixed homology class: existence of knotted
  {L}agrangian tori.
\newblock {\em J. Differential Geom.}, 74(3):507--522, 2006.

\bibitem{Ta96}
C.~H. Taubes.
\newblock {${\rm SW}\Rightarrow{\rm Gr}$}: from the {S}eiberg-{W}itten
  equations to pseudo-holomorphic curves.
\newblock {\em J. Amer. Math. Soc.}, 9(3):845--918, 1996.

\bibitem{Wu12}
W.~Wu.
\newblock On an exotic Lagrangian torus in $\mathbb{C}P^2$.
\newblock arXiv:1201.2446, 2012.

\bibitem{Wu13}
W.~Wu
\newblock Uniqueness of Lagrangian spheres in $A_n$-singularities.
\newblock Work in progress.

\end{thebibliography}
\end{document}